\documentclass{article}

\usepackage{amscd,amsthm,amsfonts,amsopn,amssymb,mathtools, amsmath}
\usepackage{enumitem}
\usepackage{fullpage}
\usepackage{comment}
\usepackage{tikz}
\usepackage{hyperref}

\usepackage{xcolor}
\usepackage{bbm}
\usetikzlibrary{external}
\newtheorem{thm}{Theorem}[section]

\newtheorem{lemma}[thm]{Lemma}
\newtheorem{prop}[thm]{Proposition}
\newtheorem{cor}[thm]{Corollary}

\newcommand{\m}[0]{\mathcal{M}}

\newcommand{\eq}{\operatorname{eq}}

\newcommand{\rb}{\operatorname{rb}}

\newcommand{\ord}{\operatorname{ord}}

\usepackage{tikz}

\title{Rainbow-Free Colorings and Rainbow Numbers for $x-y=z^k$}
\author{Katie Ansaldi \and Gabriel Cowley \and  Eric Green  \and  Kihyun Kim  \and JT Rapp}

\date{\today}

\begin{document}

\maketitle

\abstract{ An exact r-coloring of a set $S$ is a surjective function $c:S \rightarrow \{1, 2, \ldots,r\}$. A rainbow solution to an equation over $S$ is a solution such that all components are a different color. We prove that every 3-coloring of $\mathbb{N}$ with an upper density greater than $(4^s-1)/(3 \cdot 4^s)$ contains a rainbow solution to $x-y=z^k$.   The rainbow number for an equation in the set $S$ is the smallest integer $r$ such that every exact $r$-coloring has a rainbow solution. We compute the rainbow numbers of $\mathbb{Z}_p$ for the equation $x-y=z^k$, where  $p$ is prime and $k\geq 2$. } 
\section{Introduction} 
Given a set $S$, a coloring is a function that assigns a color to each element of $S$. While Ramsey theory is the study of the existence of monochromatic subsets, anti-Ramsey theory is the study of rainbow subsets. A subset $X \subseteq S$ is a \emph{rainbow} subset if each element in $X$ is assigned a distinct color. 
For example, in the equation $x_1+x_2=x_3$, a \emph{rainbow solution} is a solution $\{a,b, a+b\}$ in a set, for instance $\mathbb{Z}$ or $\mathbb{Z}_n$, such that each of $a$, $b$, $a+b$ are assigned a distinct color. A coloring is said to be rainbow-free for an equation if no rainbow solutions exist. Several papers have looked at the existence of rainbow-free 3 colorings for linear equations over $\mathbb{Z}$ and $\mathbb{Z}_n$ in \cite{JungicVeselinLicht}, \cite{Axenovich}, \cite{LlanoMontejano},  and \cite{HuicocheaMontejano}. In \cite{Zhan}, Zhan studied the existence of rainbow-free colorings for the equation $x-y=z^2$ over $\mathbb{Z}$ with certain density conditions.

The rainbow number of $S$ for $\eq$ denoted $\rb(S,\eq)$ is the smallest number of colors such that for every exact $\rb(S,\eq)$-coloring of $S$, there exists a rainbow solution to $\eq$. 
Several papers have looked at rainbow numbers over $\mathbb{Z}_n$.  For instance, the authors in \cite{Butler} looked at anti-van der Waerden numbers over both $\mathbb{Z}$ and $\mathbb{Z}_n$. In \cite{Bevilacqua}, the authors considered rainbow numbers of the equation $x_1+x_2=kx_3$ in $\mathbb{Z}_p$. The authors in \cite{Ansaldi} computed rainbow numbers of linear equations $a_1x_1 + a_2x_2 +a_3x_3=b$ over $\mathbb{Z}_n$. 
In \cite{Fallon}, they consider rainbow numbers of $x_1+x_2=x_3 $ over subsets $[m] \times [n]$ of $\mathbb{Z} \times \mathbb{Z}.$

In this paper, we generalize the results of Zhan in\cite{Zhan} to classify rainbow-free 3-colorings for the equation $x-y=z^k$ for $k \geq 2$.  We also compute the rainbow-number of $\mathbb{Z}_n$ for $x-y=z^k$.

In Section 2, we establish some preliminary notation and prove results on rainbow solutions to the equation $x-y=z^k$ in a 3-coloring of the natural numbers. The first result extends Theorem 1 of \cite{Zhan} to equations of the form $x-y=z^k$ for $k\geq 2$. 
In Section 3 we show the existence of rainbow solutions to the equation $x-y=z^k$ in three-colorings of the natural numbers that satisfy a density condition on the sizes of the color classes. In Section 4, we consider the modular case. We establish bounds on the color classes in rainbow-free colorings of $\mathbb{Z}_n$. We then establish a connection between rainbow-free colorings and the function digraph for the function $f(x) = x^k$ whose edges are of the form $(x,x^k)$. Such digraphs have been well-studies by many people including \cite{Blanton}, \cite{DresdenTu}, \cite{Lucheta}, \cite{SomerKrizek}, and \cite{Wilson}.   
   Using these digraphs, we compute the rainbow numbers of $\mathbb{Z}_p$ for $x-y=z^k$ when $k$ is prime.

\section{Rainbow-free $3$-colorings of $x-y=z^k$ in $\mathbb{N}$}
\label{section:3colorextend}

In this section, we employ the same approach as Zhan to extend \cite{Zhan} 
 Theorem 1 for the quadratic equation $x-y=z^2$ to equations of the form $x-y=z^k$, where $k\geq 2$.

Let $c: \mathbb{N} \rightarrow \{R, B, G\}$ be a 3-coloring of the set of natural numbers, where $\mathcal{R}, \mathcal{B}, \mathcal{G}$ denote the color classes of red, blue, and green, respectively. For any subset $S$ of $\mathbb{N}$ and any $n \in \mathbb{N}$, define $S(n)= |[n] \cap  S|$; hence $\mathcal R(n) = |[n] \cap \mathcal R|$. We define $\mathcal B$ and $\mathcal G$ in a similar way. A \emph{rainbow solution} to the equation $x-y= z^k$ with respect to the coloring $c$ is any ordered triple of positive integers $(a_1, a_2, a_3)$, all of different colors  such that $a_1 -a_2 = a_3^k$. We say that $c$ is  \emph{rainbow-free} for $x-y=z^k$ if there is no rainbow solution to the equation $x-y=z^k$ with respect to $c$.

A string of length $\ell$ at position $i$ consists of numbers $i$, $i+1$, $i+2, \ldots, i+\ell -1$ where $i, \ell \in \mathbb{N}$. A string is \emph{monochromatic} if it contains only one color. Similarly a string is \emph{bichromatic} if it contains exactly two colors. We say that a color is \emph{dominant} if every bichromatic string contains that color. We see that if a dominant color exists for a 3-coloring it must be unique and nondominant colors cannot be adjacent.
A string is an $R$-monochromatic string of length $\ell$ at position $i$   if $c(i-1)$, $c(i+\ell) \neq R$. We similarly define $B$-monochromatic strings and $G$-monochromatic strings. 

The following lemmas are direct extensions of \cite{Zhan} to the equation $x-y=z^k$. We include their proofs here for completeness.

\begin{lemma}
\label{lemma:reddominant}
Let $c: \mathbb{N} \rightarrow [r]$ be an exact rainbow-free coloring for $x-y=z^k$. Then $c(1)$ is dominant. 
\end{lemma}

\begin{proof}
Without loss of generality, assume $c(1) $ is red. It suffices to show that if $c(i) \neq c(i+1)$, then either $c(i)=R $ or $c(i+1)=R$. Since $(i+1, i, 1)$ is a solution to $x-y=z^k$, either $c(i) =R$ or $c(i+1)=R$. Thus, red is a dominant color, as desired. 

\end{proof}

Throughout the rest of the paper we will assume that $R$ is a dominant color. 
Here we establish that  that $B$- or $G$-monochromatic strings remain monochromatic after moving $j^k$ positions.

\begin{lemma} 
\label{lemma:nondominantstring}
Let $c: \mathbb{N} \rightarrow \{R, G, B\}$ be rainbow-free for $x-y=z^k$ with dominant color $R$. If $c(j)= B$, then for any monochromatic string as position $i$ of length $\ell$ of color $G$, the string of position $i \pm j^k$ of length $\ell$ is monochromatic of color either $B$ or $G$. 
\end{lemma} 

\begin{proof}
Suppose $c(j)=B$ and  $c(i) =c(i+1) = \ldots = c(i+ \ell -1) = G$.
For all $0 \leq h \leq \ell-1$, $(i +h + j^k, i+h, j)$ is a solution to $x-y=z^k$, and so $c(i+h + j^k ) \in \{B, G\}$. As $B$ and $G$ are nondominant colors, the string at position $i+j^k$ of length $\ell$ must be monochromatic of either blue or green. Since $(i+h, i+h-j^k, j)$ is a solution to $x-y=z^k$, a similar argument shows that the strings at position $i-j^k $ of length $\ell$ are monochromatic of color either $B$ or $G$. 
\end{proof}

\begin{lemma} [\cite{Zhan}, Lemma 5]
\label{Lemma:smalldistance}
Suppose some set $S\subset \mathbb{N}$ satisfies $\displaystyle \lim_{n\to\infty} \sup (S(n)-\frac{n}{n_0})=\infty$ for some integer $n_0 \geq 2$. then there exists a $d \leq n_0-1$ such that for any $i$, there exists $j>i$ such that $j$ and $j+d$ are both elements of $S$.
\end{lemma}

\begin{lemma}
\label{lemma: monostringbounded}
Let $c: \mathbb{N} \rightarrow \{R, G, B\}$ be  rainbow-free for $x-y=z^k$  such \[ \displaystyle \limsup_{n\rightarrow \infty}\left ({\min \{ \mathcal{R}(n),\mathcal{B}(n),\mathcal{G}(n) \}- \frac{n}{n_0}} \right )=\infty 
\]  for some integer $n_0 \geq 2$. Then the length of every nondominant monochromatic string is bounded above.
\end{lemma}

\begin{proof}
Since the upper density of every color class is finite, there cannot be any monochromatic strings of infinite length.

Suppose that $R$ is the dominant color in $c$. 
Let $i_0= \min \{ i \in \mathbb{N}  \, |  \, c(i) = B\}$. Assume for the sake of contradiction that there exist $G$-monochromatic strings of arbitrary length. Suppose there exists a $G$-monochromatic string at position $j \geq i_0$ of length $\ell$ such that $\ell \geq i_0^k$. Without loss of generality, let $j$ be the first green element in the string so that $c(j-1) = R$. Then $c(j + i_0^k-1)=G$. It follows that $(j+i_0^k -1, j-1, i_0)$ is a rainbow solution to the equation $x-y=z^k$, which is a contradiction. 
Hence, the lengths of nondominant monochromatic strings in $c$ are bounded above.

\end{proof}

An infinite arithmetic progression with initial term $i$ and common difference $d$ is monochromatic if $c(i)= c(i+d) = c(i+2d) =\ldots$. An element $j \in \mathbb{N}$ has the \emph{$A$-property}  if $c(j)$ is nondominant and there exists a monochromatic infinite arithmetic progression of the other nondominant color with common difference $j^k$.

\begin{lemma}
\label{lemma: relprimeAproperty} Let $c: \mathbb{N} \rightarrow \{R, G, B\}$ be  rainbow-free for $x-y=z^k$ with dominant color $R$.   If $c(j_1) = c(j_2) \neq R$ and $\gcd(j_1, j_2) =1$, then at most one of $j_1$ and $j_2$ has the $A$-property. 
\end{lemma}

\begin{proof}
Suppose $j_1$ and $j_2$ satisfy $\gcd(j_1, j_2) = 1$ and $c(j_1) =c(j_2) = G$. Furthermore, assume by way of contradiction that both $j_1$ and $j_2$ have the $A$-property. Then there exist two $B$-monochromatic infinite arithmetic progressions, one with initial term $i_1$ and common difference $j_1^k$ and the other with initial term $i_2$ and common difference $j_2^k$.

Suppose the $B$-monochromatic string at position $i_1$ has length $\ell_0$. By Lemma \ref{lemma:nondominantstring}, there exist $B$-monochromatic strings of length $\ell_0$ at all integers of the form $i_1 + mj_1^k$ and $i_2 +mj_1^k$ for all non-negative integers  $m$. Since $\gcd(j_1^k, j_2^k) = 1$, there exist positive integers $u_1$, $u_2$ such that $u_1j_1^k - u_2j_2^k = i_2-i_1$, and so  $i_1 + u_1j_1^k = i_2 + u_2j_2^k$. This gives a common value in both arithmetic progressions; call this common value $i_3 = i_1 + u_1j_1^k = i_2 + u_2j_2^k$.

Consider the $B$-monochromatic string at position $i_3$ of length $\ell_1$. Since $i_3$ is part of both $B$-monochromatic arithmetic progressions, by Lemma \ref{lemma:nondominantstring} there exist $B$-monochromatic strings of length $\ell_1$ at positions $i_3 + mj_1^k$ and $i_3 + mj_2^k$ for all non-negative $m$. Since $\gcd(j_1, j_2) =1$, there exist integers $v_1$ and $v_2$ such that $v_1 j_1^k - v_2j_2^k = 1$. Thus $(i_3 + v_1j_1^k) -(i_3 + v_2j_2^k) =1$, so $i_3 + v_1j_1^k= 1+ i_3 + v_2j_2^k$, and $c(i_3 + v_1j_1^k) = c(i_3 + v_2j_2^k) = B$.

Applying Lemma $\ref{lemma:nondominantstring}$ again, there exists a $B$-monochromatic string of length $\ell_1$ at position $i_3 + v_2j_2^k$.  Since $i_3 + v_1j_1^k= 1+ i_3 + v_2j_2^k$, there exists a $B$-monochromatic string of length $\ell_1+ 1$ at $i_3 + v_1j_1^k$. This process can be repeated to obtain arbitrarily long $B$-monochromatic strings, contradicting Lemma \ref{lemma: monostringbounded}. Therefore $j_1$ and $j_2$ cannot both have the $A$-property. 

\end{proof}

Again following the approach of  \cite{Zhan}, we introduce some notation here that will be used in the following lemma. Define the magnitude function $M(u,v,w, i, D) = i+ ud_1^k + vd_2^k + wd_3^k$, where $u, v, w, i \in \mathbb{Z}$ and $D= (d_1, d_2, d_3)$. Let $m$ be the length of the lattice path $P = \{(a_{p, 1},b_{p,1}), (a_{p_2}, b_{p_2}), \ldots, (a_{p, m}, b_{p,m})$. 

The following theorem gives that when there are two pairs of relatively prime elements in each nondominant color, a rainbow solution must exist. 
\begin{thm}\label{thm:relativelyprimegcd}
Let $c: \mathbb{N} \rightarrow \{R, G, B\}$ be a coloring with a dominant color $R$, satisfying the density condition of Lemma \ref{lemma: monostringbounded}.  Assume that  both $\mathcal{B}$ and $\mathcal{G}$  contain a pair of relatively prime integers. Then $c$ contains a rainbow solution to $x-y=z^k$. 
\end{thm}

\begin{proof}
Assume by contradiction that $c$ is a rainbow-free coloring for $x-y=z^k$. Suppose that $i$, $i_1 \in \mathcal B$ with $\gcd(i, i_1) = 1$ and $j$, $j_1 \in \mathcal G$ with $\gcd (j, j_1)=1$. According to Lemma \ref{lemma: relprimeAproperty}, at most one of $i$ and $i_1$ have the $A$-property, and at most one of $j$ and $j_1$ have the $A$-property. Without loss of generality, assume that $i$ and $j$ do not have the $A$-property. Since $\gcd(i_1, i, j) = 1$, there exist integers $u_0$, $v_0$, $w_0$ such that $u_0  >0$ and $v_0$, $w_0 <0$ such that $u_0 i_1^k +v_0i^k +w_0j^k = -1$. Let $D= (i_1, i, j)$.

Choose $i_2$ such that $c(i_2) \neq R$ and $i_2 > \max \{ -v_0i^k, -w_0j^k\} \geq i$. Let $\ell$ be the length of the nondominant monochromatic string at position $i_2$. Construct a lattice path as follows: let $(\alpha_0, \beta_0) = (0,0)$ and recursively define 
$(\alpha_t, \beta_t)$ by

\[(\alpha_{t+1}, \beta_{t+1}) =\begin{cases} (\alpha_t+1, \beta_t) &\text{ if } c(M(\alpha_t, 0, \beta_t, i_2, D)=G\\ (\alpha_t, \beta_t+1) &\text{ if } c(M(\alpha_t, 0, \beta_t, i_2, D)=B.
\end{cases} \]
By Lemma \ref{lemma:nondominantstring},  at each $M(\alpha_t, 0, \beta_t, i_2, D)$ there exists a nondominant monochromatic string of length $\ell$. Suppose that $\alpha_t < u_0$ for all $t$. Then for some $t$, the string of $\beta_t$ must increase infinitely many times consecutively, so $M(\alpha_t, 0, \beta_t, i_2, D) =i_2 + \alpha_ti_1^k  + \beta_t j^k$ are all blue for $t$ sufficiently large. 
This gives an infinitely long blue monochromatic sequence with  common difference of $M(\alpha_{t+1}, 0, \beta_{t+1}, i_2, D)- M(\alpha_t, 0, \beta_t, i_2, D)= j^k$. Since $c(j) = G$, this contradicts that $j$ does not have has the $A$-property. 
Therefore, there exists some $t_0$ such that $\alpha_{t_0}=u_0$. Let $q_1=\beta_{t_0}$ and so we consider the point \( (\alpha_{t_0}, \beta_{t_0} )=(u_0, v_0)  \) in the $uw$-plane. By Lemma \ref{lemma:nondominantstring}, there exists a monochromatic nondominant string of length $\ell' \geq \ell$  at position $M(u_0, 0, q_1, i_2, D)= i_2 + u_0i_1^k + q_1 j^k$. 

Construct another lattice path  in the $wv$-plane $P_0$ as follows. Let $(v_{P_0,1},w_{P_0,1}) = (0, q_1)$. Recursively define $(v_{P_0,t},w_{P_0,t})$ by

\[(v_{P_0,t+1},w_{P_0,t+1} )  =\begin{cases} (v_{P_0,t}+1,w_{P_0,t}) &\text{ if } c(M(u_0, v_{P_0,t}, w_{P_0,t}, i_2,  D)=G\\ (v_{P_0,t}, w_{P_0,t}-1) &\text{ if } c(M(u_0, v_{P_0,t}, w_{P_0,t}, i_2,  D)=B.
\end{cases} \]

By construction of $P_0$ and Lemma \ref{lemma:nondominantstring}, there exist monochromatic nondominant strings of length $\ell'$ at all $M(u_0, v_{P_0,t}, w_{P_0,t}, i_2,  D)$. Since $i$ does not have the $A$-property, there does not exist an infinite green arithmetic progression with common difference $i^k$. Therefore, there does not exist $t'$ such that for all $t>t'$, $(v_{P_0,t+1},w_{P_0,t+1})=(v_{P_0,t}+1,w_{P_0,t}) $. Thus  there must exist $q_1-w_0$ integers $t$ such that $(v_{P_0,t+1},w_{P_0,t+1})=(v_{P_0,t},w_{P_0,t}-1) $.
Therefore there exists some point  $(v_{P_0,m_0},w_{P_0,m_0}) $ on $P_0$ where $w_{P_0,m_0}=w_0$. We terminate $P_0$ at this point. Note that for all $1 \leq t \leq m_0$, we have $M(u_0,v_{P_0,t},w_{P_0,t}, i_2, D) =i_2 +u_0i_1^k +v_{P_0,t} i^k + w_{P_0,t} j^k   >0$ since $u_0,  > 0$, $v_{P_0,t} \geq 0$ and $i_2+ w_{P_0, t} j^k \geq i_2 + w_0j^k  \geq 0$ by construction of  $P_0$ and choice of $i_2$. Then by Lemma \ref{lemma:nondominantstring} and our construction of $P_0$, there are nondominant strings of length $\ell'$ at each position $M(u_0,v_{P_0,t},w_{P_0,t}, i_2, D)$ for $1 \leq t \leq m_0$. 

We construct another path $P_0'$ as follows. Let 
$(v_{P_0',1},w_{P_0',1}) = (0, q_1)$. 
Recursively define 
\[(v_{P_0',t+1},w_{P_0',t+1} )  =\begin{cases} (v_{P_0',t}-1,w_{P_0',t}) &\text{ if } c(M(u_0, v_{P_0',t}, w_{P_0',t}, i_2,  D)=G\\ (v_{P_0',t}, w_{P_0',t}+1) &\text{ if } c(M(u_0, v_{P_0',t}, w_{P_0',t}, i_2,  D)=B. 
\end{cases} \]

We again use Lemma \ref{lemma:nondominantstring} to conclude using the above method that there exists a nondominant string of length at least $\ell'$ at all positions of the form  $M(u_0,v_{P_0',t},w_{P_0',t}, i_2, D) $ whenever 
$M(u_0,v_{P_0',t},w_{P_0',t}, i_2, D) >0$ which will be satisfied as long as $v_{P_0', t} > v_0$.  Since $j$ does not have the $A$-property, at some point $m_0'$, we have $v_{P_0',m_0'}=v_0$. Terminate $P_0' $ at the point $(v_{P'_0,m_0'},w_{P'_0,m_0'}) $. 

Let $P_1$ be the union of $P_0$ and $P_0'$. The path $P_1$ is connected since $(0, q_1)$ is on both paths  and has length $m_0+m_0' -1$. Define $(v_{P_1, 1}, w_{P_1, 1}) = (v_{P_0', m_0'}, w_{P_0', m_0'})$ so that $(v_{P_1, m_0+m_0'-1}, w_{P_1, m_0+m_0'-1}) = (v_{P_0, m_0}, w_{P_0, m_0})$. Define $P_1'$ to be the path in the $vw$-plane  satisfying 
$(v_{P_1', t}, w_{P_1', t}) =(v_{P_1, t}-v_0, w_{P_1, t} -w_0)  $. Note that $(v_{P_1', 1}, w_{P_1', 1})= (0, w_{P_0', m_0'}-w_0)$ and $(v_{P_1', t}, w_{P_1', t})= (v_{P_0, m_0}-v_0, 0)$.

Finally,  construct a path $P_2$ defined as follows: 
Let $(v_{P_2, 1}, w_{P_2, 1}) = (0, 0).$ Recursively define $(v_{P_0,t},w_{P_0,t})$  by

\[(v_{P_2,t+1},w_{P_2,t+1}) =  \begin{cases} (v_{P_2,t}+1,w_{P_2,t}) &\text{ if } c(M(0, v_{P_2,t}, w_{P_2,t}, i_2,  D)=G\\ (v_{P_2,t}, w_{P_2,t}+1) &\text{ if } c(M(0, v_{P_2,t}, w_{P_2,t}, i_2,  D)=B.
\end{cases}  \]
Again by Lemma, \ref{lemma:nondominantstring}, there exists a  monochromatic nondominant string of length at least $\ell'$ at all positions of the form $M(0, v_{P_2,t}, w_{P_2,t}, i_2,  D)$. 

As Figure \ref{tikzpaths} illustrates, by construction of $P_1'$ and $P_2$, there must be a point of intersection of the two paths, say $(v_0', w_0')$ with $v_0',  w_0' > 0$. Consider the corresponding point $(v_0'+v_0, w_0'+w_0) $ on $P_1$ which corresponds to  magnitude $M(u_0, v_0'+v_0, w_0'+w_0, i_2, D) =i_2 + u_0 i_1^k + v_0'i^k + v_0i^k + w_0'j^k +w_0j^k $. On $P_2$ the point $(v_0', w_0')$  corresponds to magnitude $M(0, v_0', w_0', i_0, D)= i_2 + v_0'i^k +w_0'j^k$. Subtracting the two magnitudes gives $u_0i_1^k + v_0i^k + w_0j^k=-1$ by choice of $u_0$, $v_0$, $w_0$.  Therefore  $M(u_0, v_0'+v_0, w_0'+w_0, i_2, D) $ and $M(0, v_0', w_0', i_0, D)$ are adjacent, positive and each have a nondominant string of length at least $\ell'$ in the nondominant color. Thus a  string of length at least $\ell' + 1$ exists at $M(u_0, v_0'+v_0, w_0'+w_0, i_2, D) $, which allows us to generate arbitrarily-long nondominant monochromatic strings, contradicting \ref{lemma: monostringbounded}. Thus, we conclude that $c$ contains a rainbow solution to $x-y=z^k. $

\begin{figure}[ht]
\centering
\includegraphics{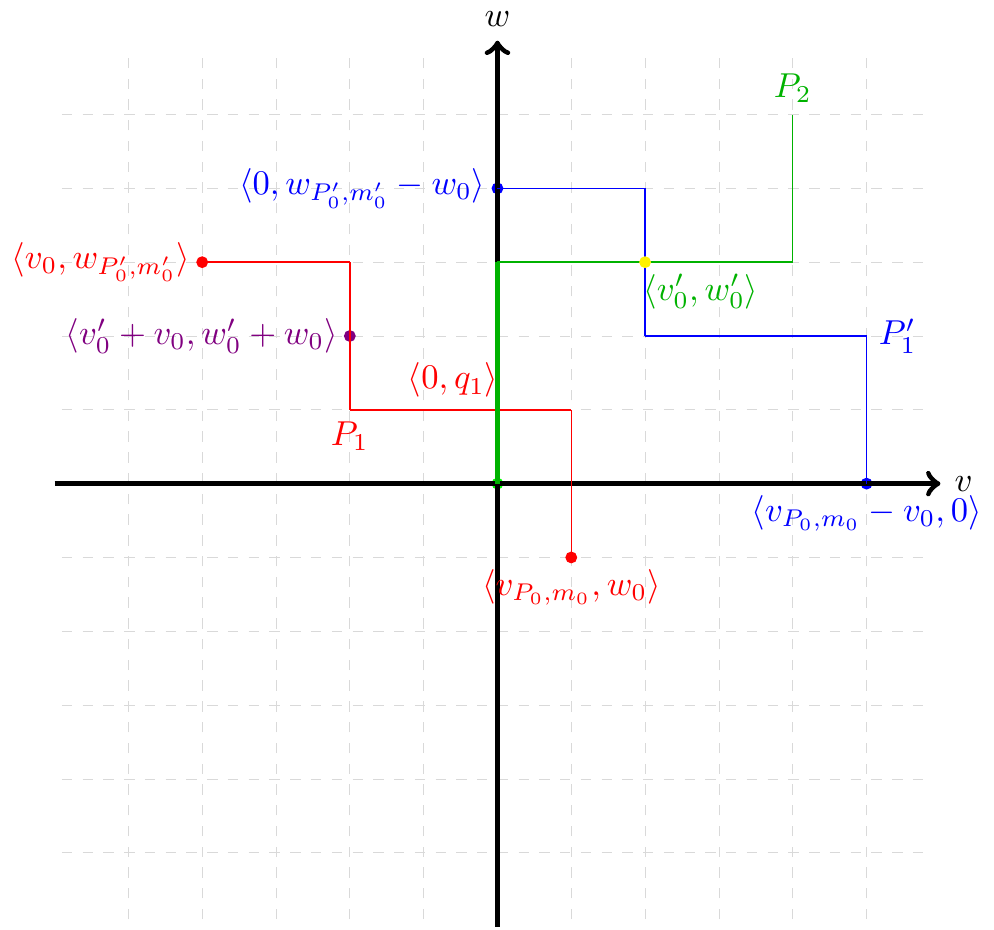}
\caption{Path $P_1, P'_1, and P_2$}
\end{figure}

\end{proof}

\section{A density condition for rainbow-free colorings  over $\mathbb{N}$}

Using Theorem \ref{thm:relativelyprimegcd}, we show that  $3$-colorings of $\mathbb{N}$ satisfying a density condition contain rainbow solutions to $x-y=z^k$. When $k=2$, the upper density is $\frac{1}{4}$ as in \cite{Zhan}.

We use the following generalization of the Frobenius coin problem in the proof of Lemma \ref{lemma:gcdnondominant}.
\begin{thm}
\label{FrobeniusCoinProblem}
Suppose two integers $i$ and $j$ satisfy gcd$(i,j)=k$. Then there exists an integer $n_0$ such that all numbers greater than $n_0$ divisible by $k$ can be written in the form $ui+vj$ for non-negative integers $u$ and $v$.
\end{thm}

In the following lemma, we use the stronger density condition to generalize \cite{Zhan}, Lemma 7. 

\begin{lemma}
\label{lemma:gcdnondominant}
Let $c: \mathbb{N} \rightarrow \{R, G, B\}$ be rainbow-free for $x-y=z^k$ such that 
\[ \displaystyle \limsup_{n\rightarrow \infty}\left ({\min \{ \mathcal{R}(n),\mathcal{B}(n),\mathcal{G}(n) \}- \frac{4^s-1}{3 \cdot 4^s}} \right )=\infty
\] 
where $s= \left \lfloor \frac{k}{2} \right \rfloor$ and  $R$ is the dominant color. Then both $\mathcal{B}$ and $\mathcal{G}$ must contain a pair of relatively prime integers.
\end{lemma}

\begin{proof}
Suppose $\mathcal B$ and $\mathcal G$ contain no pairs of consecutive integers. Since $R$ is a dominant color, for all $i$, $c(i)=R$ or $c(i+1)= R$. Then for all $n$,  $|\mathcal R(n)| \geq n/2$ and so $\liminf_{n \rightarrow \infty} (\mathcal{R}(n)-(4^s-1)/(3 \cdot 4^s)) \geq \liminf_{n \rightarrow \infty} (n/2 - (4^s-1)/(3 \cdot 4^s) ) \geq 0$. Therefore, the $\limsup_{n \rightarrow \infty}({\min\{\mathcal{B}(n),\mathcal{G}(n)\}-(4^s-1)/(3 \cdot 4^s)})\leq 0$, a contradiction. Therefore there exists an $i$ such that $i$ and $i+1$ must be in $\mathcal{B}$ or $\mathcal{G}$. Without loss of generality, suppose $i$ and $i+1$ are $B$. Then $B$ contains a pair of relatively prime integers. 

Assume $\mathcal{G}$ does not have a pair of relatively prime integers. Let $d$ be the minimum difference between any two elements in $\mathcal{G}$. Since $(4^s-1)/(3 \cdot 4^s) \geq 1/4$,  $\mathcal G$ satisfies the conditions of Lemma \ref{Lemma:smalldistance} with $n_0=4$, so we have that $d \leq 3$. Therefore, there exists a $j$ such that $j$, $j+d \in \mathcal{G}$.

First consider $d=2$.  Since $j$ and $j +2 $ are not relatively prime, $\gcd(j, j+2) =2$. There exists a $B$-monochromatic string at position $i$ of length $\ell \geq 2$. By Theorem \ref{FrobeniusCoinProblem}, there exists an integer $n_0$ such that all integers greater than $n_0$ that are divisible by $\gcd(j^k,(j+2)^k)=2^k$ can be expressed in the form $j^ku + (j+2)^kv$ for some non-negative integers $u$ and $v$. Hence,  all integers greater than $i+n_0$ that are congruent to $i \mod{2^k}$ can be expressed in the form $i+j^ku+(j+2)^kv$, and so there exist $B$-monochromatic strings at positions $i+j^k$ and $i+(j+2)^k$ of length at least 2 by Lemma \ref{lemma:nondominantstring}. By induction, for any non-negative $u$ and $v$, there exist $B$-monochromatic strings at $i+j^ku+(j+2)^kv$ of length at least 2. Thus, at some $n_1$, there exists a blue string at of length  at least 2 at all integers of the form $n_1+2^km$.

Consider a string of length $2^k$ at position $n_1+2^km$. By our assumption $c(n_1 + 2^km)= c(n_1+2^km+1)=B$. By Lemma \ref{lemma:reddominant}, $c(n_1+2^km+2) \neq G$ and $c(n_1 + 2^k(m)+2^k-1) \neq G$. 
Since $d=2$, every green element  is followed by a red element, so $|\mathcal G \cap [n_1+2^km, n_1+2^k(m+1)-1] \}| \leq  |\mathcal R \cap [n_1+2^km, n_1+2^k(m+1)-1] \}|-1 $. 
When $k$ is even, one has that $\frac{4^s-1}{3 \cdot 4^s}= \frac{2^k-1}{3 \cdot 2^k}$. Thus by the density condition for $m$ sufficiently large, we have that  $|\mathcal G \cap [n_1+2^km, n_1+2^k(m+1)-1] \}| > \frac{2^k-1}{3}$, $|\mathcal R \cap [n_1+2^km, n_1+2^k(m+1)-1] \}|> \frac{2^k-1}{3} + 1$, and $|\mathcal B \cap [n_1+2^km, n_1+2^k(m+1)-1] \}|> \frac{2^k-1}{3}$, a contradiction. When $k$ is odd, $\frac{4^s-1}{3 \cdot 4^s}= \frac{2^k-2}{3 \cdot 2^k}$. By a similar argument  we get  a contradiction here. 

Now suppose $d=3$. Since $j$ and $j +3 $ are not relatively prime, $\gcd(j, j+3) =3$. There is a monochromatic blue string of length 2 at position $i$.  By Corollary \ref{FrobeniusCoinProblem}, there exists an integer $n_2$ such that  all integers greater than $n_2$ that are divisible by $\gcd(j^k, (j+3)^k) = 3^k$ can be expressed in the form $j^ku + (j+3)^k v$. As above, there is an integer $n_3$ such that there exists a blue string of length at least 2 at all numbers of the form $n_3 + 3^km$. Since every green element is followed by at least two red elements since $d=3$, the density condition on each color class cannot hold, a contradiction. 

Therefore there exists a relatively prime pair of integers colored green.

\end{proof}

\begin{thm} 
\label{thm:densityNrainbow} 
Let $s = \lfloor k/2 \rfloor$.
Every exact $3$-coloring of the set of natural numbers with with the upper density of each color class greater than 
$(4^s-1)/(3 \cdot 4^s)$

contains a rainbow solution to $x-y=z^k$. 
\end{thm}

\begin{proof}
Assume by contradiction,  there is a rainbow-free 3-coloring $c$ of $\mathbb{N}$ for the equation $x-y=z^k$ satisfying the density condition above. By Lemma \ref{lemma:reddominant}, there exists a dominant color, say red. Since red is dominant, by Lemma \ref{lemma:gcdnondominant} , $\mathcal{B}$ and $\mathcal{G}$ each contain a pair of relatively prime integers. By Theorem \ref{thm:relativelyprimegcd}, $c$ contains a rainbow-solution, a contradiction. 

\end{proof}

\section{Rainbow numbers of $\mathbb{Z}_n$ for $x-y=z^k$}

Using the results in the previous sections on rainbow colorings over $\mathbb{Z}$, we compute rainbow numbers for $x-y=z^k $ over $\mathbb{Z}_p$.

Note rainbow-free 3-coloring of $\mathbb{Z}_n$ yield rainbow-free 3-coloring of $\mathbb{N}$. 

\begin{lemma}
\label{lemma:modtoZ}
If $\overline c: \mathbb{Z}_n \rightarrow  \{R, G, B\}$ is rainbow-free for $x-y=z^k$, then the coloring $c: \mathbb{N} \rightarrow \{R, G, B\}$ given by
$c(i) = \overline c( i \mod n)$ is rainbow-free for $x-y=z^k$. 
\end{lemma}

The following lemma is used to find pairs of relatively prime pairs in $\mathbb{N}$ in nondominant colors. The proof in \cite{Zhan} does not depend on the equation.

\begin{lemma}
\label{lemma:relprimeZn}[\cite{Zhan}, Lemma 15]
Let $\overline c: \mathbb{Z}_n \rightarrow \{R, G, B\}$ be an exact $3$- coloring of $\mathbb{Z}_n$ and let $c: \mathbb{N} \rightarrow \{R, G, B\}$ be defined by $c(i) = \overline c (i \mod n)$. 
If two integers $i_1$ and $i_2$ in $\mathbb{N}$ satisfy $\gcd(|i_1-i_2|,n)=1$, then there exists a pair of relatively prime integers $j_1$ and $j_2$ where $c(i_1)=c(j_1)$,  $c(i_2)=c(j_2)$, and $|i_1-i_2|=|j_1-j_2|$.
\end{lemma}

The following theorem generalizes \cite{Zhan} Theorem 
\begin{thm}
\label{thm: smallestcolorclass}
Let $n$ be odd and let $r_1$ be the smallest prime factor of $n$. Let $\overline c: \mathbb{Z}_n \rightarrow \{R, G, B\}$ be an exact 3-coloring of $\mathbb{Z}_n$ with corresponding color classes $\mathcal{R},\mathcal{B},\mathcal{G}$. 
If $\overline c$ is rainbow-free for $x-y=z^k$, then $\min \{ |\mathcal{R}|,|\mathcal{B}| ,|\mathcal{G}| \} \leq \frac{n}{r_1}$.

\end{thm}

\begin{proof}
Define $c: \mathbb{N} \rightarrow \{R, G, B\}$ by $c(i) = \overline c( i \mod n)$. By Lemma \ref{lemma:modtoZ},  $c$ is rainbow-free for $x-y=z^k$.   Denote the corresponding color classes of $c$ as $\mathcal R'$, $\mathcal G'$, and $\mathcal B'$.  By Lemma \ref{lemma:reddominant}, there exists a dominant color, say $R$. Suppose by contradiction that $\min \{ |\mathcal{R}|,|\mathcal{B}| ,|\mathcal{G}| \} > \frac{n}{r_1}$.  Since $\limsup_{n' \rightarrow \infty} \left(\mathcal B'(n') - \frac{n'}{r_1}  \right) =\infty $, there exists an $i_1$ and $k_1 \leq r_1-1$ such that $i_1$, $i_1+k_1 \in \mathcal B'$ by Lemma \ref{Lemma:smalldistance}. By Lemma \ref{lemma:relprimeZn}, there exists a pair of relatively prime integers $j_1$, $j_2$ where $c(j_1) =c(i_1) =B$ and  $c(j_2)=c(i_1+k_1)=B$. Similarly there exists a pair of relatively prime integers in $\mathcal G$. By Theorem \ref{thm:relativelyprimegcd}, $c$ is not rainbow-free for $x-y=z^k$, a contradiction.  
\end{proof}

For primes,  we immediately get the following corollary.

\begin{cor}
\label{cor:smallestcolorclass}
Let $p$ be prime. Let $\overline c: \mathbb{Z}_n \rightarrow \{R, G, B\}$ be an exact 3-coloring of $\mathbb{Z}_p$ with corresponding color classes $\mathcal{R},\mathcal{B},\mathcal{G}$. 
If $\overline c$ is rainbow-free for $x-y=z^k$, then $\min \{ |\mathcal{R}|,|\mathcal{B}| ,|\mathcal{G}| \} =1$.

\end{cor}

For the remainder of the section, we determine the structure of 3-colorings of $\mathbb{Z}_p$ that are rainbow-free for $x-y=z^k$ for $p$ an odd prime. 

When $p$ is prime and $a \neq 0$, the set $0, a^k, 2a^k, \ldots, (p-1)a^k $ forms a complete residue system for $\mathbb{Z}_p$. We generalize the notion of a dominant color to this complete residue system. We say that a $a^k$ string of length $\ell$ at position $ia^k$ consists of numbers $ia^k, (i+1)a^k, \ldots, (i+\ell-1) a^k$, where $i$, $\ell \in \mathbb{Z}_p$. An $a^k$-string is \emph{bichromatic} if it contains exactly two colors. A color is \emph{$a^k$-dominant} if every bichromatic string contains that color. As with dominant colors, if an $a^k$-dominant color exists for a 3-coloring it must be unique for the complete residue system $0, a^k, 2a^k, \ldots, (p-1)a^k$. 
Here we generalize Lemma  \ref{lemma:reddominant} to $a^k$-dominant colors. 

\begin{lemma}
\label{lemma:akdominant}
If $c: \mathbb{Z}_p \rightarrow \{R, G, B\}$ is an exact rainbow-free 3-coloring  for $x-y=z^k$, then $c(a)$ is $a^k$-dominant.
\end{lemma}
\begin{proof}

Without loss of generality, assume $c(a)=R$. It suffices to show that if $c(ia^k) \neq c\left((i+1)a^k\right) $, then eiter $c(ia^k) = R$ or $c\left((i+1)a^k \right)=R $. Since $c$ is rainbow-free for $x-y=z^k$ and  $\left((i+1)a^k, ia^k, a \right)$ is a solution to $x-y=z^k$, we get the desired conclusion. 

\end{proof}

\begin{lemma}
\label{lemma:negativesamecolor}
If $c: \mathbb{Z}_p \rightarrow \{R, G, B\}$ is an exact rainbow-free 3-coloring for $x-y=z^k$ then $c(a)=c(-a)$. 

\end{lemma} 
\begin{proof} Let $a \neq 0$. 
If $k$ is even, $a^k= (-a)^k$. By Lemma \ref{lemma:akdominant}, $c(a)$ is $a^k$- dominant and $c(-a)$ is $a^k$-dominant. Since $a^k$-dominant colors are unique, $c(a)=c(-a)$. 
Now consider $k$ odd. Note that if $R$ is an $a^k$-dominant color for $0, a^k, 2a^k, \ldots, (p-1)a^k$, it is also $(-a)^k$-dominant for $0, (-a)^k, \ldots, (p-1)a^k$, since the latter is the former in reverse. Since dominant colors are unique, $c(a)=c(-a)$. 

\end{proof}

The following corollary follows immediately from Corollary \ref{cor:smallestcolorclass} and Lemma \ref{lemma:negativesamecolor}.
\begin{cor}
\label{cor:0coloreduniquely}
Let $c: \mathbb{Z}_p \rightarrow \{R, G, B\}$ be an exact rainbow-free coloring for $x-y=z^k$. Suppose $c(0)=B$, then $\mathcal B=\{0\}$. That is, $0$ is the only element in its color class. 
\end{cor}

To finalize our classification of rainbow-free 3-colorings for $x-y=z^k$ over $\mathbb{Z}_p$, we consider the associated digraph from powers modulo $p$. For any function $f: \mathbb{Z}_m \rightarrow \mathbb{Z}_m$, construct a digraph  that has the elements of $\mathbb{Z}_m$ as vertices and a directed edge $(a,b)$ if and only if $f(a) \equiv b \mod \m$. 

In some cases, the digraph associated to a function $f(x)$ gives additional structure on rainbow-free colorings for the equation $x-y=f(x)$.
If $c$ is a coloring of $\mathbb{Z}_n$ and $D$ a component of $G$,  let $c(D)=\{c(a)| a \in D\}$. 
A component $D$ is \emph{monochromatic} if $|c(D)|=1$. 

\begin{lemma}
\label{lemma:monochromaticcomponent}
Let $G$ be a digraph associated to a function $f(x)$ on $\mathbb{Z}_n$  and let $D$ be a component of $G$. Let $c: \mathbb{Z}_n \rightarrow [t]$ be a rainbow-free exact $t$-coloring of $\mathbb{Z}_n$ for the equation $x-y=f(x)$.  Suppose that $c(0) \not \in C(D)$. 
Then $C$ is monochromatic. 
\end{lemma}

\begin{proof}
Suppose that $C$ is not monochromatic. 
Then there exists two adjacent vertices $a$ and $f(a)$ in $C$ such that $c(a) \neq c((f(a))$. Since $c(0) \not \in c(D)$, $(f(a), 0, a)$ is a rainbow solution to $x-y=f(x)$. 
\end{proof}

Throughout the rest of the section let $G_p^k$ be the digraph associated to the function $f(x)=x^k \mod p$. The structure of such digraphs has been well-studied, for example in \cite{SomerKrizek}, \cite{Lucheta}, \cite{Blanton}, \cite{Wilson}, and \cite{DresdenTu}.

For example, when $k=2$ and $p=10$, we have the digraph as shown in \ref{digraphexample}

\begin{figure}\label{digraphexample}

\includegraphics[scale=1]{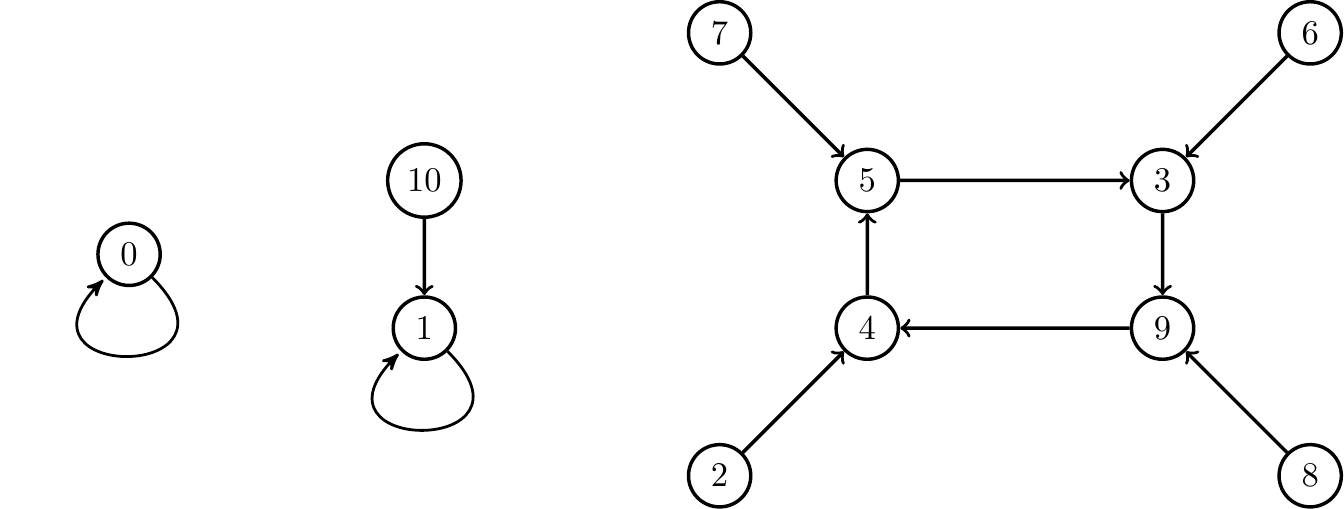}

\caption{Function digraph for $f(x)=x^2$ over $\mathbb{Z}_{11}$}
\end{figure}

We can use digraphs to  classify exact  3-colorings of $\mathbb{Z}_p$ that are rainbow-free for $x-y=z^k$. 
\begin{thm}
\label{thm:classifyrainbowfree3colorings} Let $c: \mathbb{Z}_p \rightarrow \{R, G, B\}$ be an exact  3-coloring. 
Then $c$ is rainbow-free for $x-y=z^k$ if and only if the following hold:

\begin{enumerate} \item $0$ is the only element in its color class \item  every component of $G_p^k$ is monochromatic \item $c(a)=c(-a)$ for all $a \in \mathbb{Z}_p$. 

\end{enumerate}
\end{thm}

\begin{proof}
Suppose $c$ is rainbow-free for $x-y=z^k$. Then by Corollary \ref{cor:0coloreduniquely}, $0$ is in its own color class. By Lemma \ref{lemma:monochromaticcomponent} and since $0$ is not in any other component, every component of $G_p^k$ is monochromatic. By Lemma \ref{lemma:negativesamecolor}, $c(a)=c(-a)$ for all $a \in \mathbb{Z}_p$. 

Now suppose that $0$ is the only element in its color class, every component of $G_p^k$ is monochromatic, and $c(a)=c(-a)$ for all $a \in \mathbb{Z}_p$. We show that $c$ is rainbow-free. Let $(a_1, a_2, a_3)$ be a rainbow solution to $x-y=z^k$. Then one of $a_1$, $a_2$, $a_3$ is 0 since $0$ is the only element in its color class. If $a_3=0$, then $a_1=a_2$, contradicting that $a_1$ and $a_2$ are distinct colors. If $a_2=0$, $a_1=a_3^k$, so there is a directed edge $(a_3, a_1)$ in the digraph $G_p^k$, a contradiction since the components of $G_p^k$ are monochromatic. Finally, suppose that $a_1=0$. Then $-a_2=a_3^k$. There is a directed edge $(a_3, -a_2)$, so $c(a_3)=c(-a_2)=c(a_2)$, a contradiction. Thus the coloring $c$ is rainbow-free. 

\end{proof}

As repeated iteration of $f(x)= x^k$ leads to cycles, $G_p^k $ have the following property. 

\begin{lemma} [\cite{Lucheta}, 9]
\label{lemma:componentscycles}
Every component of $G_p^k$ contains exactly one cycle.

\end{lemma}

The following theorem determines the number of components in the digraph $G_p^k$. In \cite{Lucheta}, Lucheta, Miller, and Retier consider digraphs whose vertices include only nonzero residues. We restate the theorems here for digraphs whose vertices are the elements of  $\mathbb{Z}_p$.

\begin{thm}[\cite{Lucheta}, 15] \label{theorem:cyclevertex} Let $p$ be an odd prime. 
Let $p-1=wt$, where $t$ is the largest factor of $p-1$ relatively prime to $k$. Let $c \neq 0$ be a nonzero vertex of $G_p^k$. The vertex $a$ is a cycle vertex of $G_p^k$ if and only if $\ord_p a | t$. 

\end{thm}

From this theorem \cite{Lucheta} Corollary 16 follows that there are precisely $t+1$ vertices in cycles.

Let $p$  and $t$ be as in Theorem  \ref{theorem:cyclevertex}.  Then $G_p^k$ has exactly 2 components if and only if $t=1$. Using the proposition, the prime factorization of $k$ and $p-1$ determine the number of components of $G^k_p$

\begin{prop} \label{cor: evenk2comp}  Let $k$ be even
If $k = 2^{\alpha_0} q_1^{\alpha_1}q_2^{\alpha_2} \ldots q_\ell^{\alpha_\ell}$, $q_i$ prime for $1 \leq i \leq \ell$, $\alpha_i \geq 1$, then the digraph $G^k_p$ has two components if and only if $p-1 = 2^{\beta_0}  q_1^{\beta_1}q_2^{\beta_2} \ldots q_\ell^{\beta_\ell}$ where $\beta_i \geq 0$.

\end{prop}

\begin{proof}

As a result of Corollary 16 in \cite{Lucheta}, the number of cycle vertices in $G_p^k$ is the $t$ as in the statement of Theorem \ref{theorem:cyclevertex}. It follows from that theorem that $t=1$ if and only if $p-1 = 2^{\beta_0}  q_1^{\beta_1}q_2^{\beta_2} \ldots q_\ell^{\beta_\ell}$.

Suppose $t>1$. $G_p^k$ has at least 3 cycle vertices.  Since $0^k = 0$ and $1^k =1$, there are at least 2 cycles of length 1, so there must be at least one vertex on a different cycle. Thus $G_p^k$ has more than 2 components.

If $t=1$, the only cycles are the length 1 cycles formed by $0$ and 1 so  $G_p^k$ has two components. 

\end{proof}

\begin{prop} \label{cor: oddk3comp} Let $k>3$ be odd and let    $k =  q_1^{\alpha_1}q_2^{\alpha_2} \ldots q_\ell^{\alpha_\ell}$, $q_i>2$ prime for $1 \leq i \leq \ell$, $\alpha_i \geq 1$.

Then the digraph $G_p^k$ has exactly three components if and only if $p-1 = 2^{\beta_0} q_1^{\beta_1}q_2^{\beta_2} \ldots q_\ell^{\beta_\ell}$ where $\beta_i \geq 0$.

\end{prop}

\begin{proof}
As a result of Corollary 16 in \cite{Lucheta}, the number of cycle vertices in $G_p^k$ is the $t$ as in the statement of Theorem \ref{theorem:cyclevertex}. 
We see that $t=2 $ if and only if $p-1 = 2^{\beta_0} q_1^{\beta_1}q_2^{\beta_2} \ldots q_\ell^{\beta_\ell}$ where $\beta_i \geq 0$.

Suppose $t>2$. By \ref{theorem:cyclevertex}, $G_p^k$ has at least 4 cycle vertices.  Since $0^k = 0$, $(-1)^k = -1$, and $1^k =1$, there are at least 3 cycles of length 1, so there must be at least one vertex on a different cycle. Thus $G_p^k$ has more than 3 components.

If $t=2$, the only cycles are the length 1 cycles formed by $0$, $1$, and $-1$ so  $G_p^k$ has three components. 

\end{proof}

When $k$ is odd, $-a$ may not be in the same component as $a$, but the components are symmetric. 

When an even digraph has at least three components in $\mathbb{Z}_p$, we can give a rainbow-free 3-coloring of $\mathbb{Z}_p$ by coloring the component with $0$ using one color, the component with 1 a second color, and coloring everything else a third color. For instance, in the digraph in \ref{digraphexample}, the coloring $c(0)=R$, $c(1)=c(10)=B$, and $c(2)=\ldots=c(9)=G$ gives a rainbow-free coloring of $\mathbb{Z}_{11}$ for $x-y=z^2$.

We now compute the rainbow number when $k$ is even.

\begin{thm}
Suppose $k= 2^{\alpha_0} q_1^{\alpha_1}q_2^{\alpha_2} \ldots q_\ell^{\alpha_\ell}$, $q_i$ prime for $1 \leq i \leq \ell$, $\alpha_i \geq 1$. 

\[\rb(\mathbb{Z}_p, x-y=z^k) = \begin{cases}3 & \text{ if } p-1= 2^{\beta_0}  q_1^{\beta_1}q_2^{\beta_2} \ldots q_\ell^{\beta_\ell} \text{ where }\beta_i \geq 0,  \\ 
4 & \text{ otherwise }. \end{cases}  \]
    
\end{thm}

\begin{proof}
Let $k= 2^{\alpha_0} q_1^{\alpha_1}q_2^{\alpha_2} \ldots q_\ell^{\alpha_\ell}$, $q_i$ prime for $1 \leq i \leq \ell$, $\alpha_i \geq 1$. 

Suppose  $p-1= 2^{\beta_0}  q_1^{\beta_1}q_2^{\beta_2} \ldots q_\ell^{\beta_\ell}$. By Corollary \ref{cor: evenk2comp}, the digraph $G_p^k$ has exactly two components. Let $c: \mathbb{Z}_p \rightarrow \{R, G, B\}$ be an exact 3-coloring. Since the components of $G_p^k$ are not monochromatic, by Theorem \ref{thm:classifyrainbowfree3colorings}, $c$ contains a rainbow solution to $x-y=z^k$. Thus $\rb(\mathbb{Z}_p, x-y=z^k) = 3$. 

Now suppose  
$p-1 \neq 2^{\beta_0}  q_1^{\beta_1}q_2^{\beta_2} \ldots q_\ell^{\beta_\ell}$. By Corollary \ref{cor: evenk2comp}, the digraph $G_p^k$ has at least 3 components. Define a 3-coloring $c: \mathbb{Z}_p \rightarrow \{R,G,B\}$ as follows:
\[c(a) = \begin{cases}
    R & \text{ if } a=0,\\
    B & \text{ if $a$ is in the same component as 1},\\
    G & \text{ otherwise.}
\end{cases}\]

Since $k$ is even, $(a,a^k)$ and $(-a,a^k)$ are edges in $G_p^k$. Thus $a$ and $-a$ are in the same component for all $a$. 
Since each component is monochromatic, $0$ is in its own color class, and $c(a)=c(-a)$ for all $a$, by \ref{thm:classifyrainbowfree3colorings}, $c$ does not contain a rainbow-solution to $x-y=z^k$. Thus $\rb(\mathbb{Z}_p, x-y=z^k) \geq 4$. 

Suppose $c: \mathbb{Z}_p \rightarrow \{R,B, G,Y\}$ is an exact 4-coloring. Construct  an exact 3-coloring $\overline c: \mathbb{Z}_p  $ by combining the color class that contains 0 with another color class. Since $0$ is not in its own color class in $\overline c$, by \ref{thm:classifyrainbowfree3colorings}, $\overline c$ contains a rainbow solution  to $x-y=z^k$. By construction $c$ also contains a rainbow solution and so every exact 4-coloring contains a rainbow solution. Thus $\rb(\mathbb{Z}_p, x-y=z^k)\leq 4$. 

\end{proof}
As a consequence the rainbow-number of $\mathbb{Z}_p$ for $x-y=z^2$ is 3 if and only if $p$ is a Fermat prime.

When an odd digraph has at least three components in $\mathbb{Z}_p$, we can give a rainbow-free 3-coloring of $\mathbb{Z}_p$ by coloring the component with $0$ using one color, the components containing $1$ and $p-1$ with a second color, and coloring everything else a third color.

\begin{thm}
Suppose $k=  q_1^{\alpha_1}q_2^{\alpha_2} \ldots q_\ell^{\alpha_\ell}$, $q_i \geq 3$ prime for $1 \leq i \leq \ell$, $\alpha_i \geq 1$, with $k>3$.  

\[\rb(\mathbb{Z}_p, x-y=z^k) = \begin{cases}3 & \text{ if } p-1 = 2^{\beta_0} q_1^{\beta_1}q_2^{\beta_2} \ldots q_\ell^{\beta_\ell} \text{ where }\beta_i \geq 0,  \\ 
4 & \text{ otherwise.} \end{cases}  \]
    
\end{thm}

\begin{proof}
Let $k= q_1^{\alpha_1}q_2^{\alpha_2} \ldots q_\ell^{\alpha_\ell}$, $q_i$ prime for $1 \leq i \leq \ell$, $\alpha_i \geq 1$. 

    Suppose $p-1= 2^{\beta_0} q_1^{\beta_1}q_2^{\beta_2} \ldots q_\ell^{\beta_\ell}$ where $\beta_i \geq 0.$ By \ref{cor: oddk3comp}, the digraph $G_p^k$ has 3 components. Since $1$ and $-1$ are both cycle vertices in $G_p^k$,  $1$ and $-1$ are in distinct components.  Let $c: \mathbb{Z}_p \rightarrow \{R, G, B\}$ be an exact 3-coloring. If the components are monochromatic, each component must be a distinct color. In particular, $c(1) \neq c(-1)$, so \ref{thm:classifyrainbowfree3colorings} shows that there exists a rainbow solution to $x-y=z^k$ in $c$. Otherwise, the components are not chromatic, and again there is a rainbow solution to $x-y=z^k$. Thus $\rb(\mathbb{Z}_p, x-y=z^k) =3$.

    Now suppose  
$p-1 \neq 2^{\beta_0}  q_1^{\beta_1}q_2^{\beta_2} \ldots q_\ell^{\beta_\ell}$. By Corollary \ref{cor: oddk3comp}, the digraph $G_p^k$ has at least 4 components and  $1$ and $-1$ are in distinct components. Define a 3-coloring $c: \mathbb{Z}_p \rightarrow \{R,G,B\}$ as follows:
\[c(a) = \begin{cases}
    R & \text{ if } a=0,\\
    B & \text{ if $a$ is in the same component as 1 or -1},\\
    G & \text{ otherwise.}
\end{cases}\]

Suppose that $(a, a^k) $ is an edge in the component containing 1. Then $(-a, -a^k) $ is an edge in the component containing $-1$. Thus $c(a)= c(-a)$ for all $a \in \mathbb{Z}_p$. Furtheremore component is monochromatic and $0$ is in its own color class, and so by  \ref{thm:classifyrainbowfree3colorings}, $c$ does not contain a rainbow-solution to $x-y=z^k$. Thus $\rb(\mathbb{Z}_p, x-y=z^k) \geq 4$.

Let $c: \mathbb{Z}_p \rightarrow {R,B, G,Y}$ be an exact 4-coloring. We can create a new exact 3-coloring $\overline{c}: \mathbb{Z}_p$ by combining the color class that contains 0 with another color class. Since $0$ is not in its own color class in $\overline{c}$, according to \ref{thm:classifyrainbowfree3colorings}, $\overline{c}$ contains a rainbow solution to $x-y=z^k$.

By construction, $c$ also contains a rainbow solution. Therefore, we can conclude that every exact 4-coloring contains a rainbow solution, and thus $\rb(\mathbb{Z}_p, x-y=z^k)\leq 4$. Hence, we have shown that $\rb(\mathbb{Z}_p, x-y=z^k)= 4$.
\end{proof}

\bibliographystyle{plain} 
\bibliography{main} 

\begin{thebibliography}{10}

\bibitem{Ansaldi}
Katie Ansaldi, Houssein El~Turkey, Jessica Hamm, Anisah Nu'Man, Nathan
  Warnberg, and Michael Young.
\newblock Rainbow numbers of {$\mathbb Z_n$} for {$a_1x_1 + a_2x_2 + a_3x_3 =
  b$}.
\newblock {\em Integers}, 20:Paper No. A51, 15, 2020.

\bibitem{Axenovich}
Maria Axenovich and Dmitri Fon-Der-Flaass.
\newblock On rainbow arithmetic progressions.
\newblock {\em Electron. J. Combin.}, 11(1):Research Paper 1, 7, 2004.

\bibitem{Bevilacqua}
Erin Bevilacqua, Samuel King, J\"{u}rgen Kritschgau, Michael Tait, Suzannah
  Tebon, and Michael Young.
\newblock Rainbow numbers for {$x_1 + x_2 = kx_3$} in {$\mathbb Z_n$}.
\newblock {\em Integers}, 20:Paper No. A50, 27, 2020.

\bibitem{Blanton}
Earle~L. Blanton, Jr., Spencer~P. Hurd, and Judson~S. McCranie.
\newblock On a digraph defined by squaring modulo {$n$}.
\newblock {\em Fibonacci Quart.}, 30(4):322--334, 1992.

\bibitem{Butler}
Steve Butler, Craig Erickson, Leslie Hogben, Kirsten Hogenson, Lucas Kramer,
  Richard~L. Kramer, Jephian Chin-Hung Lin, Ryan~R. Martin, Derrick Stolee,
  Nathan Warnberg, and Michael Young.
\newblock Rainbow arithmetic progressions.
\newblock {\em J. Comb.}, 7(4):595--626, 2016.

\bibitem{DresdenTu}
Greg Dresden and Wenda Tu.
\newblock Finding cycles in the {$k$}-th power digraphs over the integers
  modulo a prime.
\newblock {\em Involve}, 11(2):181--194, 2018.

\bibitem{Fallon}
Kean Fallon, Ethan Manhart, Joe Miller, Hunter Rehm, Nathan Warnberg, and Laura
  Zinnel.
\newblock Rainbow numbers of $[m] \times [n]$ for $x_1 + x_2 = x_3$, 2023.

\bibitem{HuicocheaMontejano}
Mario Huicochea and Amanda Montejano.
\newblock The structure of rainbow-free colorings for linear equations on three
  variables in {$\mathbb{Z}_p$}.
\newblock {\em Integers}, 15A:Paper No. A8, 19, 2015.

\bibitem{JungicVeselinLicht}
Veselin Jungi\'{c}, Jacob Licht, Mohammad Mahdian, Jaroslav Ne~et il, and
  Rado\v{s} Radoi\v{c}i\'{c}.
\newblock Rainbow arithmetic progressions and anti-{R}amsey results.
\newblock {\em Combin. Probab. Comput.}, 12(5-6):599--620, 2003.
\newblock Special issue on Ramsey theory.

\bibitem{LlanoMontejano}
Bernardo Llano and Amanda Montejano.
\newblock Rainbow-free colorings for {$x+y=cz$} in {$\mathbb{Z}_p$}.
\newblock {\em Discrete Math.}, 312(17):2566--2573, 2012.

\bibitem{Lucheta}
Caroline Lucheta, Eli Miller, and Clifford Reiter.
\newblock Digraphs from powers modulo {$p$}.
\newblock {\em Fibonacci Quart.}, 34(3):226--239, 1996.

\bibitem{SomerKrizek}
Lawrence Somer and Michal K\v{r}\'{\i}\v{z}ek.
\newblock The structure of digraphs associated with the congruence {$x^k\equiv
  y\pmod n$}.
\newblock {\em Czechoslovak Math. J.}, 61(136)(2):337--358, 2011.

\bibitem{Wilson}
Brad Wilson.
\newblock Power digraphs modulo {$n$}.
\newblock {\em Fibonacci Quart.}, 36(3):229--239, 1998.

\bibitem{Zhan}
Tong Zhan.
\newblock On rainbow solutions to an equation with a quadratic term.
\newblock {\em Integers}, 9:A49, 655--670, 2009.

\end{thebibliography}

\end{document}